\documentclass[12pt]{amsart}
\usepackage{amsmath, amssymb, amsthm, mathrsfs, textcomp, colonequals, comment,soul}
\usepackage{epic}
\usepackage[shortlabels]{enumitem}
\usepackage{xypic}
\usepackage[backref]{hyperref}
\usepackage[alphabetic,backrefs,lite]{amsrefs}
\usepackage[usenames,dvipsnames]{color}
\usepackage{fullpage}

\hypersetup{
  colorlinks   = true, 
  urlcolor     = blue, 
  linkcolor    = blue, 
  citecolor   = blue 
}

\numberwithin{equation}{section}
\allowdisplaybreaks[1]
\newtheorem{theorem}[equation]{Theorem}

\newtheorem{corollary}[equation]{Corollary}
\newtheorem{prop}[equation]{Proposition}
\newtheorem{lemma}[equation]{Lemma}

\theoremstyle{definition}
\newtheorem{assumption}[equation]{Assumption}
\newtheorem{remark}[equation]{Remark}

\newtheorem{defn}[equation]{Definition}

\newtheorem{example}[equation]{Example}
\newtheorem{question}[equation]{Question}

\setenumerate{itemsep=2pt}

\newcommand{\ol}[1]{\overline{#1}}
\newcommand{\mc}[1]{\mathcal{#1}}
\newcommand{\mf}[1]{\mathfrak{#1}}

\newcommand{\Q}{\mathbb Q}

\newcommand{\Z}{\mathbb Z}

\renewcommand{\P}{\mathbb P}
\newcommand{\proj}{\mathbb P}

\renewcommand{\phi}{\varphi}

\newcommand{\Spec}{\mathrm{Spec} \ }

\DeclareMathOperator{\Gal}{Gal}

\DeclareMathOperator{\GL}{GL}
\DeclareMathOperator{\divi}{div}
\DeclareMathOperator{\disc}{disc}
\DeclareMathOperator{\sep}{sep}

\DeclareMathOperator{\odd}{odd}
\DeclareMathOperator{\even}{even}
\DeclareMathOperator{\cha}{char}
\DeclareMathOperator{\Art}{Art}
\DeclareMathOperator{\Aut}{Aut}

\DeclareMathOperator{\Hom}{Hom}

\DeclareMathOperator{\db}{db}

\DeclareMathOperator{\length}{length}

\title[Conductor-discriminant inequality]{Conductor-discriminant inequality for hyperelliptic curves in odd residue characteristic}

\author{Andrew Obus}
\address{Baruch College}
\curraddr{1 Bernard Baruch Way. New York, NY 10010, United States of America.}
\email{andrewobus@gmail.com}

\author{Padmavathi Srinivasan}
\address{ICERM, Brown University}
\curraddr{121 South Main Street, 11th Floor, Providence, RI 02903, United States of America.}
\email{padmavathi\_srinivasan@brown.edu}


\date{\today}

\begin{document}

\maketitle

\begin{abstract}
We prove an inequality between the conductor and the discriminant for all hyperelliptic
curves defined over discretely valued
fields $K$ with perfect residue field of characteristic not $2$.  Specifically,
if such a curve is given by $y^2 = f(x)$ with $f(x) \in \mc{O}_K[x]$,
and if $\mc{X}$ is its minimal regular model over
$\mc{O}_K$, then the negative of the Artin conductor of $\mc{X}$ (and
thus also the number of irreducible components of the special fiber of
$\mc{X}$) is bounded
above by the valuation of $\disc(f)$.  There are no
restrictions on genus of the curve or on the ramification of the splitting field of
$f$.  This generalizes earlier work of Ogg, Saito, Liu, and the second author.  
\end{abstract}

\section{Introduction}
In this note, we prove a conductor-discriminant inequality 
for all hyperelliptic curves over discretely valued fields with perfect residue field of characteristic not $2$.

\subsection{Main theorem}\label{Smain}
Let $K$ be a discretely valued field with perfect residue field $k$ of characteristic not $2$. Let $\mc{O}_K$ be the ring of integers of $K$. Let $\nu_K \colon K \rightarrow \Z \cup \{ \infty
\}$ be the corresponding discrete valuation. Let $X$ be a smooth,
projective, geometrically integral curve of genus $g \geq 1$ defined
over $K$.  Let $\mc{X}$ be a proper,
flat, regular $\mc{O}_K$-scheme with generic fiber $X$. The \emph{Artin conductor} associated to the model $\mc{X}$ is defined by
\[ \Art (\mc{X}/\mc{O}_K) =  \chi(\mc{X}_{\overline{K}}) - \chi(\mc{X}_{\overline{k}}) - \delta ,\]
where $\chi$ is the Euler characteristic for the $\ell$-adic cohomology and $\delta$ is the Swan conductor associated to the $\ell$-adic representation $\Gal (\overline{K}/K) \rightarrow \Aut_{\Q_\ell}
(H^1_{\mathrm{et}}(\mc{X}_{\overline{K}}, \Q_\ell))$ ($\ell \neq \cha k$). The Artin conductor is a measure of degeneracy of the model $\mc{X}$; it is a non-positive integer that is zero precisely when
$\mc{X}/\mc{O}_K$ is smooth or when $g=1$ and $(\mc{X}_k)_{\mathrm{red}}$ is smooth. If $\mc{X}/\mc{O}_K$ is a regular, semistable model, then $-\Art(\mc{X}/\mc{O}_K)$ equals the number of singular points of the special fiber $\mc{X}_k$. 

For hyperelliptic curves, there is another measure of degeneracy
defined in terms of minimal Weierstrass equations. Assume that $X$ is
hyperelliptic, with hyperelliptic degree $2$ morphism $X \to Y \cong \proj^1_K$. An integral Weierstrass equation for $X$ is an
equation of the form $y^2= f(x)$ with $f(x) \in \mc{O}_K[x]$, such that $X$ is birational to the plane curve given by this equation. The discriminant of such an equation is defined to be the
non-negative integer $\nu_K(\disc'(f))$, where $\disc'(f)$ is the discriminant of $f$, thought of as a polynomial of degree $2 \lceil \deg(f)/2 \rceil$ (note that this is the usual discriminant $\disc(f)$
whenever $f$ is monic or $\deg(f)$ is even). The main theorem of the paper is the folllowing.
\begin{theorem}\label{Preduce} Let $K$ be the fraction field of a Henselian discrete valuation ring with algebraically closed residue field of characteristic not $2$ and let $f \in \mc{O}_K[x]$ be a separable polynomial with $\deg(f) \geq 3$. Let $X$ be the hyperelliptic curve with affine equation $y^2=f(x)$. Then there exists a proper flat regular $\mc{O}_K$-model $\mc{X}_f$ of $X$ such that 
\begin{equation}\label{Econddiscineq} -\Art(\mc{X}_f/\mc{O}_K) \leq \nu_K(\disc'(f)). \end{equation}
\end{theorem}
We call (\ref{Econddiscineq}) the \textit{conductor-discriminant inequality for $f$}.

A \emph{minimal Weierstrass equation} is an equation for which the integer $\nu_K(\disc'(f))$ is as small as possible amongst all integral equations. We define
the \emph{minimal discriminant} $\Delta_{X/K}$ of $X$ to be
$\nu_K(\disc'(f))$ for the minimal Weierstrass equation. The minimal
discriminant of $X$ is zero precisely when the minimal proper regular
model of $X$ is smooth over $S$. Let $\Art (X/K)$ denote the Artin conductor associated to the
\emph{minimal} proper regular model of $X$ over $\mc{O}_K$.  

When $g=1$, we have $-\Art (X/K) = \Delta_{X/K}$ by the Ogg-Saito
formula \cite[p.\ 156, Corollary 2]{saito2}. When $g=2$,
Liu~\cite[p.\ 52,
Th\'eor\`eme 1 and p.\ 53, Th\'eor\`eme 2]{Li:cd} shows that $-\Art
(X/K) \leq \Delta_{X/K}$; he also shows that equality can fail to
hold. In the second author's thesis \cite{PadmaRational}, Liu's
inequality was extended to hyperelliptic curves of arbitrary genus
assuming that the roots of $f$ are defined over an \emph{unramified}
extension of $K$. In subsequent work \cite{PadmaTame}, the second
author proved the same inequality assuming only that roots of $f$ are
defined over a \emph{tame} extension of $K$. The argument in \cite{PadmaTame} is an induction on a natural combinatorial gadget attached to a polynomial called the metric tree that records the $p$-adic distances between the roots of the polynomial. 

As a corollary to Theorem~\ref{Preduce}, we prove this inequality for \emph{all} cases away from residue characteristic $2$. 
\begin{corollary}\label{Tfinalthm}
 Let $X$ be a hyperelliptic curve of genus $g \geq 1$ over a discretely valued field $K$ with perfect residue field of characteristic not
 equal to $2$. Let $\Delta_{X/K}$ be the minimal discriminant of $X$
 and let $\Art(X/K)$ denote the Artin conductor of the minimal regular
 model of $X$. Then
 $-\Art(X/K) \leq \Delta_{X/K}$.
\end{corollary}
\begin{proof}
We may assume that $K$ is Henselian, since the invariants in (\ref{Econddiscineq}) are constant under unramified base change and regular models satisfy \'{e}tale descent.  

Let $\mc{X}/\mc{O}_K$ be a regular model of $X$. Let $n$ be the number of irreducible components of the geometric special fiber $\mc{X}_{\ol{k}}$ and let 
$\phi$ be the conductor exponent for the Galois representation $\Gal
(\overline{K}/K) \rightarrow \Aut_{\Q_\ell}
(H^1_{\mathrm{et}}(X_{\overline{K}}, \Q_\ell))$ ($\ell \neq \cha k$),
which only depends on $X$. Then \cite[Proposition~1]{Li:cd} shows that
\begin{equation}\label{Partinconductor}
 -\Art(\mc{X}/\mc{O}_K) = n - 1 + \phi.
\end{equation}

If $\mc{X}$ is a proper regular
model of $X$, then the number of irreducible
components of $\mc{X}_{\ol{k}}$ is at least the number of irreducible
components in the geometric special fiber of the minimal regular model
of the curve $X$.  Thus (\ref{Partinconductor}) implies $-\Art(X/K)
\leq -\Art(\mc{X}/\mc{O}_K)$.
The minimal discriminant of a hyperelliptic curve $X$ is equal
to the discriminant of one of the integral polynomials $f$ that
defines it via an equation $y^2 = f(x)$.  So if $f$ is such a
polynomial, we have 
\begin{equation}\label{E:eqstring}-\Art(X/K) \leq \-\Art(\mc{X}_f/\mc{O}_K) \leq
\nu_K(\disc'(f)) = \Delta_{X/K},\end{equation} where the second inequality is
Theorem~\ref{Preduce}.  This proves the corollary.
\end{proof}

\begin{remark}
The proof of Corollary~\ref{Tfinalthm} in this paper in fact gives a new proof
of the results in \cite{PadmaRational} and \cite{PadmaTame}. 
\end{remark}

\begin{prop}\label{P:genus1}
 Keep the notation of Theorem~\ref{Preduce}. Suppose $\deg(f)=3$. Then 
 \[ -\Art(\mc{X}_f/\mc{O}_K) = \Delta_{X/K}.\]
\end{prop}

\begin{remark}\label{R:genus1close}
 Let $\mc{X}^{\mathrm{min}}$ be the minimal proper regular model of an elliptic curve $X$. The Ogg-Saito formula is the assertion that $-\Art(\mc{X}^{\mathrm{min}}) = \Delta_{X/K}$. By Proposition~\ref{P:genus1}, away from residue characteristic $2$, the Ogg-Saito formula is equivalent to the assertion that the canonical map $\mc{X}_f \rightarrow \mc{X}^{\mathrm{min}}$ is an isomorphism. 
 \end{remark}

\subsection{Related work of other authors}\label{Searlier}
\subsubsection{Small genus}\label{Ssmallgenus}
 In genus $1$, the proof of the Ogg-Saito formula used the explicit
 classification of special fibers of minimal regular models of genus
  $1$ curves. In genus $2$, \cite{Li:cd} defines another discriminant that is
  specific to genus $2$ curves, and compares both the Artin conductor
  and the minimal discriminant (our $\Delta_{X/K}$, which Liu calls $\Delta_0$) to
  this third discriminant (which Liu calls $\Delta_{\min}$). This
  third discriminant $\Delta_{\min}$ is sandwiched between the Artin
conductor and the minimal discriminant and is defined using a possibly
non-integral Weierstrass equation such that the associated
differentials generate the $\mc{O}_K$-lattice of global sections of
the relative dualizing sheaf of the minimal regular model.  It does
not directly generalize to higher genus hyperelliptic curves (but see
\cite[Definition~1, Remarque~9]{Li:cd} for a related
conductor-discriminant question).
Liu even provides an explicit formula for the difference between the
Artin conductor and both $\Delta_0$ and $\Delta_{\min}$ that can be described in terms of the
combinatorics of the special fiber of the minimal regular model (of
which there are already over $120$ types!).   This leads one to ask
the following question, which we do not address in this paper.

\begin{question}
Can one give an interpretation of the difference between $-\Art(X/K)$ and $\Delta_{X/K}$ in Corollary~\ref{Tfinalthm}, analogous to the interpretation given in \cite{Li:cd}?
\end{question}

\subsubsection{General curves}\label{Srelated}

Several people have worked on comparing conductor exponents and discriminants. In the semistable case, work of Kausz \cite{Kau} (when $p \neq 2$) and Maugeais \cite{Mau} (all $p$) compares
the Artin conductor to yet another notion of discriminant. In \cite{DDMM}, the authors compute many arithmetic invariants attached to hyperelliptic curves in the semistable case in terms of the cluster
picture of the polynomial $f$ (which encodes the same information as
the metric tree of the roots of $f$.) In \cite{Kohls}, Kohls compares the conductor exponent $\phi$ with the minimal discriminant of superelliptic
curves, by studying the Galois action on the special fiber of the semistable model as in \cite{BW_Glasgow}. In \cite{BKSW}, the authors define minimal discriminants of Picard curves (degree $3$
cyclic covers of $\P^1_K$) and compare the conductor exponent and the minimal discriminant for such curves.

\subsection{Summary of proof strategy}\label{Sideas}
Assume for the rest of the introduction that $\deg(f)$ is even, so
$\disc'(f) = \disc(f)$.  The common technique of \cite{PadmaRational},
\cite{PadmaTame}, and this paper is to build a regular
model $\mc{X}_f$ of $X$ by normalizing a 
specific regular model $\mc{Y}_f$ of $Y \cong \mathbb{P}^1_K$ in $K(X)$.  The model $\mc{Y}_f$ is
an embedded resolution of $(\proj^1_{\mc{O}_K}, B)$, where $B$ is the
branch locus of the normalization of the standard model
$\proj^1_{\mc{O}_K}$ in $K(X)$.  That is, $\mc{Y}_f$ is a blowup of
$\P^1_{\mc{O}_K}$ on which all components of $\divi(f)$ of odd
multiplicity are regular and disjoint. 

In \S\ref{Sexcess}, we reduce the proof of the
conductor-discriminant inequality to an inequality between the number
of components of the model $\mc{Y}_f$ and the ``discriminant
bonus''

\begin{equation}\label{Efirstdb}
  \db_K(f) \colonequals \nu_K(\disc(f)) - \sum_{i=1}^r \nu_K(\disc(K_i/K)),
\end{equation}
where $f = f_1 \cdots f_r$ is an irreducible factorization in $K[x]$ and $K_i$ is the field generated by a root of $f_i$. Namely, Remark~\ref{Rreduction} says that $-\Art(\mc{X}_f/\mc{O}_K) \leq \nu_K(\disc(f))$ if and only if
\begin{equation}\label{Efirstreformulation}
  2(N_{\mc{Y}_f,\even} - 1) \leq \db_K(f),
\end{equation}
where $N_{\mc{Y}_f,  \even}$ is the number of irreducible components of the special fiber
of $\mc{Y}_f$ on which the order of $f$ is even
(see Proposition~\ref{Partincalculation}).

The main content of \S\ref{Sinduction}, where Theorem~\ref{Preduce}
and Proposition~\ref{P:genus1} are proved, is an inductive argument 
that shows that we can build $\mc{Y}_f$ by blowing up successive points on
models $\mc{Y}$ of $\proj^1_K$ where the branch locus of $\mc{Y}$ in $K(X)$ is singular, and that the
inequality (\ref{Efirstreformulation}) is satisfied at the end of this
process. We heartily thank the referee of a previous submission for
the core of this argument.  

\begin{example}\label{Ebasic}
Consider the hyperelliptic curve $X$ given by the affine
equation $y^2 = f(x)$, where $f(x) = x^d - \pi_K$ and $\pi_K$ is a
uniformizer of $K$.
In this case, the normalization $\mc{X}$ of $\proj^1_{\mc{O}_K}$ (with
coordinate $x$) in the
function field $K(X)$ is already regular.

Assume $d$ is even for simplicity.  Then $\chi(\mc{X}_{\ol{K}}) = 4 -
d$.  On the other hand, the special fiber of $\mc{X}$ is given by the affine
equation $y^2 = x^d$, so it is a union
of two copies of $\proj^1_k$ meeting at one point.  Thus
$\chi(\mc{X}_{k}) = 2 - 0 + 1 = 3$.  So
$-\Art(\mc{X}/\mc{O}_K) = d - 1 + \delta$, where $\delta$ is the Swan conductor.  Using, e.g.,
Proposition~\ref{Pswanconductor}, one calculates $\delta
= \nu_K(d)$.  We also have $\nu_K(\disc (f)) = \nu_K(d) + d - 1$.  Thus the conductor-discriminant inequality is an equality in this
case.

Note that the special fiber of $\mc{X}$ does \emph{not} have simple
normal crossings when $d \geq 4$, since the irreducible components do
not meet transversely.  By Equation~\ref{Partinconductor},
the minimal snc-model $\mc{X}'$ of $X$ has $-\Art(\mc{X}'/\mc{O}_K) >
-\Art(\mc{X}/\mc{O}_K) = \disc(f)$, which means that $\mc{X}'$ does \emph{not}
satisfy the conductor-discriminant inequality.  So minimal snc-models
are insufficient for our purposes. 
\end{example}

\section*{Notation and conventions}
Throughout, $K$ is a Henselian field with respect to a
discrete valuation $\nu_K$ with residue characteristic not $2$.   We further assume
that the residue field $k$ of $K$ is \emph{algebraically closed}.
We denote \emph{fixed} separable and algebraic
closures of $K$ by $K^{\sep} \subseteq \ol{K}$.  All algebraic
extensions of $K$ are assumed to live inside $\ol{K}$.  This means that
for any algebraic extension $L/K$, there is a preferred embedding
$\iota_L \in \Hom_K(L, \ol{K})$, namely the inclusion.
We fix a uniformizer $\pi_K$ of $\nu_K$
and normalize $\nu_K$ so that $\nu_K(\pi_K) = 1$.

For a finite separable field extension $L/K$, we let $\disc(L/K)$
denote the discriminant of the field extension $L/K$ and let
$\Delta_{L/K} \colonequals \nu_K(\disc(L/K))$. For any separable
polynomial $f \in K[x]$, we let $\disc(f)$ (resp.\ $\disc'(f)$) denote the discriminant of
the polynomial $f$ viewed as a polynomial of degree
$\deg(f)$ (resp.\ degree $2 \lceil \deg(f)/2 \rceil$) and let
$\Delta_{f,K} \colonequals \nu_K(\disc(f))$. Note that with this
convention, if $f=cg$ for some monic polynomial $g$, then
$\Delta_{f,K} = 2\nu_K(c)(\deg(g)-1)+\Delta_{g,K}$.
We will suppress the index $K$ whenever the field is clear.

For an integral $K$-scheme or $\mc{O}_K$-scheme $S$, we denote the corresponding function
field by $K(S)$. If $\mc{Y} \rightarrow \mc{O}_K$ is an arithmetic
surface, an irreducible codimension 1 subscheme of $\mc{Y}$ is called
\emph{vertical} if it lies in a fiber of $\mc{Y} \to \mc{O}_K$, and
\emph{horizontal} otherwise. Let $f \in K(\mathcal{Y})$. We denote the
divisor of zeroes of $f$ by $\divi_0(f)$. If $\divi(f) = \sum_i m_i
\Gamma_i$, call a component $\Gamma_i$ for which $m_i$ is odd an
\emph{odd component} of $\divi(f)$ on $\mathcal{Y}$. Similarly define
\emph{even component} of $\divi(f)$ (this includes every component
$\Gamma_i$ for which $m_i = 0$).

If $C$ and $D$ are divisors on a regular, proper, flat relative curve
over a Dedekind scheme, we write $(C, D)$ for their intersection number.

\section*{Funding}
AO was supported by National Science Foundation grants DMS-1602054, DMS-1900396 and DMS-2047638, by a grant from the Simons Foundation (\#706038, AO), as well as by a PSC-CUNY Award, jointly funded by the Professional Staff Congress and the City University of New York. PS is supported by the Simons Collaboration in Arithmetic Geometry, Number Theory, and Computation via the Simons Foundation grant 546235.

\section*{Acknowledgements}
The proof of inequality (\ref{Efirstreformulation}) was originally a
much more involved argument that spanned over 40 pages explicitly
describing the regular model $\mathcal{Y}_f$, see \cite{OS1}
and \cite{OS2}.  We thank
the anonymous referee of \cite{OS1} for showing us the
argument of Proposition~\ref{Pinduction}, allowing us to drastically
shorten the paper.

We also acknowledge the hospitality of the
Mathematisches Forschungsinstitut Oberwolfach, where we participated in the
``Research in Pairs'' program during the writing of this paper.
Lastly, we would like to thank Dino Lorenzini for useful conversations.

\section{The discriminant bonus and regular models}\label{Sexcess}
Recall that $X/K$ is a hyperelliptic curve with affine equation $y^2 =
f(x)$, where $X \to \proj^1$ is the projection to the $x$-coordinate.  The \emph{discriminant} of such an equation is the
integer $\nu_K(\disc'(f))$.  Changing $x$-coordinates on $\P^1_K$ using an element of $\GL_2(\mc{O}_K)$ does not change the valuation of the discriminant of an equation.  Since $k$ is algebraically closed, we may assume that $f$ has even degree by such a change of coordinates (\cite[Section~1.3]{PadmaRational}), and we may even assume that no root of $f$ specializes to $\infty$.  That is, we may assume that all roots of $f$ lie in $\mc{O}_K$.  If $f$ has repeated roots, then $\disc'(f) = 0$ and (\ref{Econddiscineq}) is satisfied automatically, so assume also that $f$ is separable.  Lastly, since $K$ is Henselian with algebraically closed residue field of characteristic not $2$, the group $K^{\times}/(K^{\times})^2$ has two elements, whose coset representatives are $1$ and $\pi_K$.  So after multiplying $f$ by squares, which does not change the isomorphism class of $X$, we may assume that the leading coefficient of $f$ is $1$ or $\pi_K$.
Thus, for the remainder of the paper, we make the following assumption:

\begin{assumption}\label{Afform}
The polynomial $f(x)$ has even degree, is separable, and has irreducible factorization
$\pi_K^bf_1 \ldots f_r$, where the $f_i \in \mc{O}_K[x]$ are monic irreducible polynomials and $b \in \{0, 1\}$.
\end{assumption}

The argument above proves the following proposition.
\begin{prop}\label{Preduce2}
If the conductor-discriminant inequality (Theorem~\ref{Preduce}) holds for all $f$ satisfying Assumption~\ref{Afform}, then it holds for all $f \in \mc{O}_K[x]$.
\end{prop}

For $f$ satisfying Assumption~\ref{Afform}, we define $K_i =
K[x]/f_i(x)$ for $1 \leq i \leq r$.

\begin{prop}\label{Pswanconductor}
The Swan conductor of $X$ equals $\sum_{i=1}^r (\Delta_{K_i/K} - \deg
f_i + 1) = r - \deg(f) + \sum_{i=1}^r \Delta_{K_i/K}$.
\end{prop}
\begin{proof}
The argument in \cite[Theorem 1.20(i)]{DDMM} for $K$ a local field
works also for $K$ Henselian discretely valued with algebraically closed residue field, with the added simplification that all residue degrees are $1$.
\end{proof}

\begin{defn}\label{Ddiscbonus}
The \emph{discriminant bonus of $f$ over $K$}, written $\db_K(f)$, is the
quantity $\Delta_{f, K} - \sum_{i=1}^r \Delta_{K_i/K}$. 
\end{defn}

\begin{remark}\label{Rdiscprop}
  If $f = \pi_K^b f_0$ where $f_0$ is monic, then
$\Delta_{f,K} = \Delta_{f_0,K} + 2b(\deg(f) - 1)$, so 
\begin{equation*}\label{Edbdecomp}
\db_K(f) = \db_K(f_0) + 2b(\deg(f) - 1).  
\end{equation*}
\end{remark}

\begin{remark}\label{Rcolength}
If $f_0$ is monic, then $\db_K(f_0) = \length_{\mc{O}_K} (d_{C/A}/d_{B/A})$, where $B = \mc{O}_K[x]/f_0$ and $C$ is the integral closure of $B$ in its total ring of fractions.
As a consequence of \cite[III, \S2, Proposition 5]{SerreLF}, we get $\db_K(f_0) =
2\length_{\mc{O}_K}(C/B)$. 
\end{remark}

We now obtain a formula for the Artin conductor. Let $\mc{Y}$ be a
regular model of $\P^1_K$ and let $\mc{X}$ be the normalization of
$\mc{Y}$ in $K(x)[y]/(y^2-f(x))$. Let $B$ be the branch locus of
$\mc{X} \to \mc{Y}$. Write $\mc{Y}_s$ and $\mc{Y}_{\ol{\eta}} \cong
\proj^1_{\ol{K}}$ for the special and geometric generic fibers of
$\mc{Y}$, respectively, and write $B_s$ and $B_{\ol{\eta}}$ for the
special and geometric generic fibers of $B$, respectively.  Let
$N_{\mc{Y}}$ be the number of irreducible components of $\mc{Y}_s$,
and let $N_{\mc{Y},\odd}/N_{\mc{Y},\even}$ be the
number of odd/even vertical components of $\divi(f)$, so $N_{\mc{Y}} =
N_{\mc{Y}, \odd} + N_{\mc{Y}, \even}$.
\begin{prop}\label{Partincalculation} Keep the notation from the paragraph above. Assume that the odd components of $\divi_0(f)$ are regular and pairwise disjoint. Then $\mc{X}$ is regular and we have
$$-\Art(\mc{X}/\mc{O}_K) = 2(N_{\mc{Y}} - 1) - 2N_{\mc{Y},\mathrm{odd}} +
\sum_{i=1}^r \Delta_{K_i/K} = 2(N_{\mc{Y},\mathrm{even}}-1) + \sum_{i=1}^r \Delta_{K_i/K}$$
\end{prop}
\begin{proof}
By \cite[Lemma~2.1]{PadmaRational}, the model $\mc{X}$ is regular. By \cite[Lemma 2.2]{PadmaTame}, we have
$$-\Art(\mc{X}/\mc{O}_K) = 2(\chi(\mc{Y}_s) - \chi(\mc{Y}_{\ol{\eta}})) -
(\chi(B_s) - \chi(B_{\ol{\eta}})) + \delta,$$ where $\delta$ is the
Swan conductor of $X$. We will use $H^i$ and $h^i$ to denote the \'{e}tale cohomology groups and their dimensions respectively. Now, $\mc{Y}_s$ and
$\mc{Y}_{\ol{\eta}}$ both have trivial $H^1$ and one-dimensional
$H^0$, while $h^2(\mc{Y}_s) = N_{\mc{Y}}$ and
$h^2(\mc{Y}_{\ol{\eta}}) = 1$.  So $\chi(\mc{Y}_s) - \chi(\mc{Y}_{\ol{\eta}})
= N_{\mc{Y}} - 1$.  Since $\deg(f)$ is even by
Assumption~\ref{Afform}, it follows that $B$ consists
of precisely all the odd components of $\divi_0(f)$. Since the odd components of $\divi_0(f)$ are regular and pairwise disjoint, it follows that as a closed subset, $B_s$ is a disjoint union of closed points and closed codimension $1$ sets: the closed points correspond to points where the horizontal components of $\divi_0(f)$ specialize, so there is exactly one for each irreducible factor of $f$, and the codimension $1$ sets correspond to the vertical components appearing with odd multiplicity in $\divi(f)$ on $\mathcal{Y}$. By \cite[Lemma~7.1]{ObusWewers}, these irreducible components are all isomorphic to $\P^1_k$ and therefore have trivial $H^1$, and since $\chi$ is an additive functor, it follows that $\chi(B_s) = r + 2N_{\mc{Y},\mathrm{odd}}$.  Since $\deg(f)$ is even, $B_{\ol{\eta}}$ consists of $\deg(f)$ points and therefore $\chi(B_{\ol{\eta}}) = \deg(f)$.  
Lastly, by Proposition \ref{Pswanconductor}, $\delta = r - \deg(f) +
\sum_{i=1}^r \Delta_{K_i/K}$.  Putting everything together proves the proposition.
\end{proof}

\begin{remark}\label{Rreduction} In light of
  Proposition~\ref{Partincalculation} and Definition~\ref{Ddiscbonus} of the
discriminant bonus, in
order to prove the conductor-discriminant inequality for $f$
satisfying Assumption~\ref{Afform}, it suffices to
find a regular model $\mc{Y}_f$ of $\P^1_K$ on which the odd components of
$\divi_0(f)$ are regular and disjoint, such that
\begin{equation}\label{Ereformulation}
\db_K(f) \geq 2(N_{\mc{Y}_f,\mathrm{even}} - 1).
\end{equation} 
We say that such a model $\mc{Y}_f$ \emph{realizes the
  conductor-discriminant inequality for $f$}.
If the inequality in \eqref{Ereformulation} is an
equality, then the normalization $\mc{X}_f$ of $\mc{Y}_f$ in $K(X)$
satisfies $-\Art(\mc{X}_f/\mc{O}_K) = \Delta_{f,K}$, so we
say that $\mc{Y}_f$ \emph{realizes the conductor-discriminant
  equality} for $f$. 
\end{remark}

\section{Proof of inequality \ref{Ereformulation}}\label{Sinduction}

\subsection{Preliminary lemmas}
We begin with a pair of prelliminary lemmas about divisors on
arithmetic surfaces.

\begin{lemma}\label{Laddlength}
 Suppose that $C'$ and $C''$ are reduced, relatively prime, effective divisors on a
regular proper flat relative curve $\mc{X}$
over $\mc{O}_K$.  
If $C = C' + C''$ and $\widetilde{C}$, $\widetilde{C}'$, $\widetilde{C}''$ are the
respective normalizations of $C$, $C'$, $C''$, we have that
$$\length_{\mc{O}_K}(\mc{O}_{\widetilde{C}}/\mc{O}_C) =
\length_{\mc{O}_K}(\mc{O}_{\widetilde{C}'}/\mc{O}_{C'}) +
\length_{\mc{O}_K}(\mc{O}_{\widetilde{C}''}/\mc{O}_{C''}) + (C',C'').$$
\end{lemma}

\begin{proof}
Since $\mc{O}_{\widetilde{C}} \cong \mc{O}_{\widetilde{C}'} \times
\mc{O}_{\widetilde{C}''}$, it suffices to show that
\begin{equation}\label{Elengthcomparison}
  \length_{\mc{O}_K} (\mc{O}_{C'} \times \mc{O}_{C''} / \mc{O}_C) = (C',
  C'').
\end{equation}
It suffices to check this locally at each point $P$
of $C' \cap C''$.  If $A = \mc{O}_{\mc{X},P}$, then both $C'$ and
$C''$ are principal in $\Spec A$,
so let $C' = \divi(g')$ and $C'' = \divi(g'')$.  Then $C = \divi(gg')$
and 
(\ref{Elengthcomparison}) follows from the exact sequence
$$0 \to A/gg' \to A/g' \times A/g'' \to A/(g', g'') \to 0$$ and the
fact that the local intersection number $(C', C'')_P$ equals
$\length_{\mc{O}_K}(A/(g', g''))$.
\end{proof}

The following lemma is an adaptation of \cite[V, Proposition
3.7]{Hartshorne} and \cite[Ex.\ 9.2.12]{LiuBook} to the case of an arithmetic surface with a possibly
horizontal divisor.

\begin{lemma}[{cf.\ \cite[Prop.\ 3.7]{Hartshorne}}]\label{Lmult}
  Let $C$ be an effective divisor on a regular, proper, flat relative
  curve $\mc{X}$ over $\mc{O}_K$.  Let $\pi \colon \mc{X}' \to \mc{X}$ be the
  blowup at a multiplicity $\mu$ closed point $x$ of $C$, and let $C'
  \to C$ be the strict transform.  Then
  $\length_{\mc{O}_K}(\mc{O}_{C'}/\mc{O}_C) = \mu(\mu-1)/2$.  
\end{lemma}

\begin{proof}
For this proof, we may restrict to an affine neighbourhood $U \colonequals \Spec A$ of $x$ in  $\mc{X}$ such that
the maximal ideal $m_{U, x}$ is generated by two global functions $u$
and $v$ on $U$, and by shrinking $U$, we may assume that the curve $C$
is cut out by a single polynomial equation $\phi = 0$.  
As in the proof of \cite[Proposition 9.2.23]{LiuBook}, we may write
$\phi = P(u, v) + Q$ for a homogeneous polynomial $P$ of degree $\mu$ with coefficients in $\mc{O}_{U,x}^\times$ and for  $Q(u, v) \in m_{U, x}^{\mu + 1}$. 
Up to a linear change of generators of $m_{U, x}$, we may further assume that $w \colonequals \pi^*(u)/\pi^*(v)$ is a regular function at every point in the preimage of $x$ in $C'$.
Up to further shrinking $U$ if necessary, we may assume that $C' \cap
\pi^{-1}(U)$ is contained in the affine $U \times \mathbb{A}^1$ where
$w$ is regular, where $C'$ is cut out by the equations $wu=v$ and
$\phi/u^\mu$, and furthermore $\deg_w P(1,w) = \mu$.

Concretely, letting $g = Q(1, w)/u^{\mu}$, we have $\phi/u^{\mu} =
P(1,w) + ug$ and the ring of regular functions on $C' \cap
\pi^{-1}(U)$ corresponds to the ring $R \colonequals A[w]/(wu-v,P(1,w)
+ ug)$. Let $B \colonequals A/(\phi)$ be the coordinate ring of $C
\cap U$ and let $\mathfrak{m}$ be the image of the ideal $(u,v)
\subseteq A$ in $B$.  Then $\mf{m}$ is the maximal ideal of $x$ in
$B$, and $B/\mathfrak{m} \cong k$.  Under the natural inclusion $B
\hookrightarrow R$ given by the identity on $A$, we have
$\length_{\mc{O}_K}(\mc{O}_{C'}/\pi^{*} \mc{O}_{C}) =
\length_{\mc{O}_K} R/B$. We will now write down an explicit
composition series for the inclusion $B \subseteq R$ of $\mc{O}_K$-modules, such that the associated graded $\mc{O}_K$-module is isomorphic to $k^{\mu(\mu-1)/2}$, where $k$ is the residue field of $\mc{O}_K$. Since $\length_{\mc{O}_K}{k} = 1$, this proves the result. 

For $0 \leq i \leq \mu-1$ and $0 \leq j \leq \mu - 1 - i$, let $M^i$ be
the $B$-submodule of $R$ spanned by
$1$, $w$, $w^2$,\ldots, $w^i$, and let $M^i_j$ be the
$B$-submodule of $R$ spanned by $M^i$ and the elements
$v^{\mu - 1 - i - j}w^{i+1}$, $v^{\mu - 1 - i - (j-1)}w^{i + 1}$,\ldots, $v^{\mu - 2 - i}
w^{i+1}$ (thus, $M^i_0 = M^i$ and $M^i_{\mu-1-i} = M^{i+1}$). Since $\deg_w P(1,w) = \mu$, we also have that $M^{\mu-1} = R$.  In
particular, we have
\begin{align*}
B  = M^0 = M^0_0 \subseteq M^0_1 \subseteq \cdots
\subseteq M^0_{\mu-1} &= \\ M^1 = M^1_0 \subseteq M^1_1 \subseteq M^1_2 \subseteq \cdots
\subseteq M^1_{\mu - 2} &= \\ M^2 = M^2_0 \subseteq M^2_1 \subseteq M^2_2 \subseteq \cdots
\subseteq M^2_{\mu - 3} &= \\ 
\cdots \subseteq M^{\mu-2}_1 &= M^{\mu-1} = R. 
\end{align*}
Furthermore, for all $i \geq 0$, $j \geq 1$ with $i + j \leq \mu - 1$,
$M^i_j$ is generated over $M^i_{j-1}$ as a $B$-module by
$\alpha^i_j := v^{\mu-1-i-j}w^{i+1}$.  Since $$u\alpha^i_j = u
v^{\mu-1-i-j}w^{i+1} = v^{\mu - i - j} w^i \in M^{i-1}_j \subseteq
M^i_{j-1} \text{ and }  v\alpha^i_j = v^{\mu-i-j}w^{i+1} \in M^i_{j-1},$$ we have that
$\mathfrak{m}$ annihilates $M^i_j/M^i_{j-1}$, or equivalently that 
$M^i_j/M^i_{j-1} \cong k$. Thus we have constructed
a composition series for $R/B$ of length $\mu(\mu-1)/2$ where the successive quotients are all isomorphic to the residue field $k$ of $\mc{O}_K$. \qedhere
\end{proof}

\subsection{Proof of Theorem~\ref{Preduce}}

Now, let $\mc{Y}$ be regular, proper, flat, relative curve over $\Spec
\mc{O}_K$, and let $D \subseteq \mc{Y}$
be a reduced effective Cartier divisor. Set $(\mc{Y}_0,D_0) = (\mc{Y},D)$ and define a sequence
$(\mc{Y}_n, D_n)$, $n = 0, . . . , N$ as follows. Suppose $(\mc{Y}_{n-1}, D_{n-1})$ is defined. If $D_{n-1}$ is regular, we
set $N = n - 1$ and stop here. Suppose otherwise and let $x_n \in
D_{n-1}$ be a singular point with multiplicity $m_n$.
Let $f_n \colon \mc{Y}_n \to \mc{Y}_{n-1}$ be the blow-up at $x_n$ and $E_n = f_n^{-1}(x_n)$ be the exceptional divisor. Let
$D_n' \subseteq \mc{Y}$ be the proper transform of $D_{n-1}$. Then $f_n^{-1}(D_{n-1})= D_n' + m_nE_n$.  If $m_n$ is even (resp.\ odd), set
$D_n = D_n'$ (resp.\ $D_n = D_n' + E_n$).  Let $\widetilde{D}$ be the
normalization of $D$.

\begin{prop}\label{Pinduction}
  With notation as above, we have
  \begin{equation}\label{Edivisors}
  \#\{n \mid 0 \leq n \leq N \text{ and } m_n \text{ is even}\} \leq
  \length_{\mc{O}_K}(\mc{O}_{\widetilde{D}}/\mc{O}_D),
    \end{equation}
 with equality holding if
  and only if the multiplicity of every $x_n$ in each $D_{n-1}$ for $0 < n \leq N$ is at
  most $3$.
\end{prop}

\begin{proof}
  By induction on $n$, it suffices to show the following inequality:
$$\length_{\mc{O}_K}(\mc{O}_{\widetilde{D}_n}/\mc{O}_{D_n}) \leq \length_{\mc{O}_K}
\mc{O}_{\widetilde{D}_{n-1}}/\mc{O}_{D_{n-1}} - \begin{cases} 0 & m_n
  \text{ odd} \\ 1 & m_n \text{ even.}\end{cases},$$ 
with equality if and only if $m_n \in \{2, 3\}$. 

First, suppose $m_n$ is even. Then, $D_n = D_n'$ is the proper transform of $D_{n-1}$ and we have $\widetilde{D}_n =
\widetilde{D}_{n-1}$.
By Lemma~\ref{Lmult}, we have $\length_{\mc{O}_K} (\mc{O}_{D_n}/\mc{O}_{D_{n-1}}) = m_n(m_n - 1)/2 \geq
1$, since $m_n > 1$. So
\begin{align}
\begin{split}
  \length_{\mc{O}_K}(\mc{O}_{\widetilde{D}_n}/\mc{O}_{D_n}) &=
  \length_{\mc{O}_K}(\mc{O}_{\widetilde{D}_{n-1}}/\mc{O}_{D_n}) \\
  &= \length_{\mc{O}_K}(\mc{O}_{\widetilde{D}_{n-1}}/\mc{O}_{D_{n-1}}) - \length_{\mc{O}_K}(\mc{O}_{D_n}/\mc{O}_{D_{n-1}}) \\
  &\leq \length_{\mc{O}_K}(\mc{O}_{\widetilde{D}_{n-1}}/\mc{O}_{D_{n-1}}) - 1,
\end{split}
\end{align}
 which proves the proposition in this case.

Now, suppose $m_n$ is odd.  Recall that $D_n = D_n' + E_n$, where
$D_n$ is the strict transform of $D_{n-1}$.  By Lemma~\ref{Laddlength}
applied to $D_n'$ and the regular divisor $E_n$, we have 
\[ \length_{\mc{O}_K} (\mc{O}_{\widetilde{D}_n} / \mc{O}_{D_n}) =
\length_{\mc{O}_K} (\mc{O}_{\widetilde{D}_n'} / \mc{O}_{D_n'}) + (E_n,
D_n').\]    
By \cite[V, Corollary 3.7]{Hartshorne}\footnote{The result in \cite{Hartshorne} is stated only for
  projective surfaces, but the proof goes through verbatim in the
  arithmetic surface case.}, we have $(E_n, D_n') = m_n$.  
Since $D_n'$ is the strict transform of $D_{n-1}$, we have
$\widetilde{D}_n' = \widetilde{D}_{n-1}$.  Putting all this together once again with Lemma~\ref{Lmult} yields 
\begin{align}
\begin{split}
  \length_{\mc{O}_K} (\mc{O}_{\widetilde{D}_n} / \mc{O}_{D_n}) &=
  \length_{\mc{O}_K} (\mc{O}_{\widetilde{D}_n'} / \mc{O}_{D_n'}) + m_n \\
  &= \length_{\mc{O}_K} (\mc{O}_{\widetilde{D}_{n-1}} / \mc{O}_{D_n'}) +
  m_n \\
  &= \length_{\mc{O}_K}(\mc{O}_{\widetilde{D}_{n-1}}/\mc{O}_{D_{n-1}}) -
    \length_{\mc{O}_K}(\mc{O}_{D_n'}/\mc{O}_{D_{n-1}}) + m_n \\
    &= \length_{\mc{O}_K}(\mc{O}_{\widetilde{D}_{n-1}}/\mc{O}_{D_{n-1}})
    - m_n(m_n-1)/2 + m_n \\
    &\geq \length_{\mc{O}_K}(\mc{O}_{\widetilde{D}_{n-1}}/\mc{O}_{D_{n-1}}),
\end{split}
    \end{align}
since $m_n \geq 3$.
Equality occurs when $m_n = 3$, proving the proposition in this case.
\end{proof}

\begin{corollary}\label{Cmain}
Let $X \to \proj^1_K$ be a hyperelliptic curve given by affine
equation $y^2 = f(x)$ satisfying Assumption~\ref{Afform}. 
In the notation above, let $\mc{Y} = \mc{Y}_0 = 
\proj^1_{\mc{O}_K}$ be the standard smooth model of $\proj^1_K$ with
coordinate $x$.  Lastly, let $D \subseteq \mc{Y}$ be the branch locus
of the normalization of $\mc{Y}$ in $K(X)$.

Then, in the notation
above, the model $\mc{Y}_N$ satisfies the hypotheses of
Remark~\ref{Rreduction}, as well as inequality (\ref{Ereformulation})
with respect to $f$.  That is, in the language of
Remark~\ref{Rreduction}, we say that $\mc{Y}_N$ realizes the
conductor-discriminant inequality of $f$. 

Furthermore, if the multiplicity $m_n$ of each $x_n$ in $D_{n-1}$ the notation above is at most $3$, then
$\mc{Y}_N$ realizes the conductor-discriminant equality for
$f$. 
\end{corollary}

\begin{proof}
Observe that $D = D_0 \equiv \divi(f) \pmod{2
\, \text{Div}(\mc{Y}_0)}$ as divisors on $\mc{Y}_0$, and the same congruence holds with $D_i$ in place of $D_0$ on
$\mc{Y}_i$. 
In particular, $D_N$ and the odd part of
$\divi(f)$ have the same support on $\mc{Y}_N$.  Since
$D_N$ is regular, the model $\mc{Y}_N$ satisfies the hypotheses of
Remark~\ref{Rreduction}.

Let $f = \pi^bf_1\cdots f_r$ with $b \in \{0, 1\}$ as in Assumption~\ref{Afform}.  In
Proposition~\ref{Pinduction}, the left hand side of (\ref{Edivisors})
counts all even
vertical components of $\divi(f)$ on $\mc{Y}_N$, except possibly the strict
transform of the special fiber $S$ of $\mc{Y}_0$.  Since this strict
transform is even for $\divi(f)$ exactly when $b = 0$, we have that the left
hand side of (\ref{Edivisors}) equals $N_{\mc{Y}_N, \text{even}} - (1 - b)$.  

Let $D_{\text{horiz}}$ be the horizontal part of $D$.  Then $D = bS +
D_{\text{horiz}}$.  Since $(S, D_{\text{horiz}}) = \deg(f)$,
Lemma~\ref{Laddlength} applied to $D_{\text{horiz}}$ and the regular
divisor $S$ implies that the
\begin{equation}\label{E:fromvertohorz} \length_{\mc{O}_K}(\mc{O}_{\widetilde{D}}/\mc{O}_D) = \length_{\mc{O}_K}(\mc{O}_{\widetilde{D}_{\text{horiz}}}/\mc{O}_{D_{\text{horiz}}}) +
b\deg(f) .\end{equation}

The sheaf $\mc{O}_{D_{\text{horiz}}}$ is the
sheafification of the $\mc{O}_K$-algebra $B := \mc{O}_K[x]/(f/\pi_K^b)$,
whereas $\mc{O}_{\widetilde{D}_{\text{horiz}}}$ is the sheafification
  of the integral closure $C$ of $B$ in its total ring of fractions.
  By Remark~\ref{Rcolength},
  $$\length_{\mc{O}_K}(\mc{O}_{\widetilde{D}_{\text{horiz}}}/\mc{O}_{D_{\text{horiz}}})
  = \length_{\mc{O}_K}(C/B)
  = (1/2)\db_K(f/\pi_K^b).$$  So, using Remark~\ref{Rdiscprop} for the first equality, 
and Proposition~\ref{Pinduction} and \eqref{E:fromvertohorz} in the inequality below,
\begin{align}\label{Edbineq}
\begin{split}
\db_K(f) &= 2b(\deg(f) - 1) + \db_K(f/\pi_K^b) \\
&= 2b(\deg(f) - 1) + 2 \length_{\mc{O}_K}(\mc{O}_{\widetilde{D}_{\text{horiz}}}/\mc{O}_{D_{\text{horiz}}}) \\
&\geq 2b(\deg(f) - 1) - 2b \deg(f) + 2(N_{\mc{Y}_N, \text{even}} - (1 - b)) \\
&= 2(N_{\mc{Y}_N, \text{even}} - 1),
\end{split}
\end{align} 
which proves the inequality (\ref{Ereformulation}). 
By Proposition~\ref{Pinduction}, equality holds in \eqref{Edbineq}
exactly when each $x_n$ has multiplicity at most $3$ in $D_{n-1}$, proving the last statement of the corollary. 
\end{proof}

\begin{proof}[Proof of Theorem~\ref{Preduce}]
 This is immediate from Proposition~\ref{Preduce2}, Remark~\ref{Rreduction}, and Corollary~\ref{Cmain}.
\end{proof}

\subsection{Proof of Proposition~\ref{P:genus1}}\label{Sequality}

\begin{defn}\label{Dminimalform}
  A reduced effective divisor $D$ on a regular arithmetic surface
  $\mc{Y}$ over $\mc{O}_K$ is \emph{robustly of
  multiplicity $\leq n$} at $P$ if it has multiplicity $\mu_{D,P} \leq n$ at
$P$, and furthermore if $\mu_{D,P} = n$, then $D$ has reducible
tangent cone at $P$.
\end{defn}

\begin{lemma}\label{Lminweierstrasspreserved}
  Let $\mc{Y}$ be a regular arithmetic surface over $\mc{O}_K$, let $D$ be a
  reduced effective divisor on $\mc{Y}$, let $P \in \mc{Y}$ be a
  closed point such that $D$ has multiplicity $\mu_{D, P}$ at $P$.  Let $E$ be the exceptional divisor of
   $\text{Bl}_P(\mc{Y}) \to \mc{Y}$, and let $D'$ be the strict
   transform of $D$ on $\text{Bl}_P(\mc{Y})$.  Let $n \geq 1$.  If $D$ is robustly of
   multiplicity $\leq n$ at $P$, then both $D'$ and $D' + E$ are robustly of
   multiplicity $\leq n$ at every closed point $P' \in D' \cap E$.
\end{lemma}

\begin{proof} 
We first claim that, since $P$ is a closed point on a regular surface, $$\sum_{Q
  \in D' \cap E} \mu_{D', Q} \leq \mu_{D,P}.$$ To prove this, note
that intersection theory on the blow up tells us that \[0 = (\pi^*D, E)
  = (D'+\mu_{D,P}E,E) = (D',E)-\mu_{D,P},\] and so $(D',E) = \mu_{D,P}$,
and it suffices to observe that for each $Q$ in $D' \cap E$, we have
$\mu_{D',Q} \leq i_Q(D',E)$, since $D'$ and $E$ are locally Cartier and the defining equation for $D'$ at $Q$ is also in $\mathfrak{m}_{E,Q}^{\mu_{D',Q}}$. In particular, for $P'
\in D' \cap E$, we have $\mu_{D', P'} \leq
\mu_{D, P},$ with strict inequality whenever the tangent cone
to $D$ at $P$ is reducible.

Let us prove the statement for $D'$.  If $\mu_{D,P} < n$ then
$\mu_{D', P'} \leq \mu_{D,P} < n$.  If
$\mu_{D, P} = n$ then $D$ has reducible tangent cone at $P$, so
$\mu_{D', P} < \mu_{D,P} = n$.  In both cases, $\mu_{D', P} < n$, so
we are done.

Now we prove the statement for $D' + E$, which always has
reducible tangent cone.  So it suffices to show that if
$P' \in D' \cap E$, then $\mu_{D' + E, P'} \leq n$, or equivalently
that $\mu_{D', P'} < n$.  This was proved in the previous paragraph.
\end{proof}

\begin{prop}\label{P:deg3}
If $\deg(f) = 3$ and $y^2=f(x)$ is a minimal Weierstrass equation,
then the process in Proposition~\ref{Pinduction} with $\mc{Y}_0=
\mathbb{P}^1_{\mc{O}_K}$ and $D_0$ the reduced divisor satisfying $D_0
\equiv \divi(f) \pmod{2 \, \text{Div}(\mc{Y}_0)}$ always yields $m_n
\in \{2,3\}$. 
\end{prop}
\begin{proof}
Let $D_n,D_n',E_n,\mc{Y}_n$ be as in Proposition~\ref{Pinduction}.
We will first use the minimality of $f$ to argue
  that every singular point of $D_0$ has multiplicity $\leq
  3$. Suppose not. Assume that $P$ is a point of multiplicity $\geq 4$
  on $D_0$, and observe that $D_0 = \divi(f)$ away from $x = \infty$
  since our Weierstrass equation is minimal. Without loss of generality, we may assume that $P$
  corresponds to the maximal ideal $(x,\pi)$. Write $f(x) =
  \sum_{i=0}^3 c_ix^i$. Since $x,\pi \in \mathfrak{m}_{P}
  \setminus \mathfrak{m}_P^2$, the assumption $\mu \geq 4$ implies
  that $\pi^{4-i} \mid c_i$ for every $i$. Letting $z \colonequals
  x/\pi$, we see that $f(x) = \pi^4 \sum (c_i/\pi^{4-i}) z^i
  \equalscolon \pi^4 \tilde{f}(z)$, with $\tilde{f}(z) \in
  \mc{O}_K[z]$. Letting $\tilde{y} \colonequals y/\pi^2$, we see that
  $\tilde{y}^2=\tilde{f}(z)$ is another integral Weierstrass equation
  for $X$ with associated discriminant $\disc(f)-6$, contradicting the
  minimality of $f$. So every singular point of $D_0$ has multiplicity
  $\leq 3$.
  
If no singular point of $D_0$ has multiplicity $3$, then all singular
points are robustly of multiplicity $\leq 3$, which implies (by
Lemma~\ref{Lminweierstrasspreserved}) that the same is true for all
$D_n$, proving the proposition.

Now, suppose there is a singular point $P$ of $D_0$ of multiplicity
$\mu = 3$ which we take to be $x_1$. If the tangent cone to $D_0$ at
$P$ is reducible, then $D$ is robustly of multiplicity $\leq 3$ at
$P$, and by
Lemma~\ref{Lminweierstrasspreserved} the same is true for all $D_n$ at
all points,
and we are done.  So assume the tangent cone to $D_0$ at $P$ is irreducible,
and assume further (without loss of generality) that $P$ corresponds to the
maximal ideal $\mathfrak{m}_P = (x,\pi)$ in $\mc{O}_K[x]$.  The irreducibility of the tangent cone means that there exists $g
\in \mathfrak{m}_P$ such that $g^3 \equiv f(x)
\pmod{\mathfrak{m}_P^4}$. We have the freedom to add elements of
$\mathfrak{m}_P^2$ to $g$, so we may assume that $g = ax + b \pi$,
where $a$ and $b$ are in $\mc{O}_K^{\times}$.  After an invertible
change of variables, we may thus assume that $g = x$, and so $$f(x)
\equiv x^3 \pmod{\mathfrak{m}_P^4}.$$  This implies that if we write
$f(x) = \sum_{i=0}^3 c_{i}x^i$ with $c_j \in \mc{O}_K$, then
$v(c_0) \geq 4$, $v(c_1) \geq 3$, $v(c_2) \geq 2$, and $v(c_3) = 0$.  On the other
hand, the minimality of the Weierstrass equation implies that either
$v(c_0) < 6$, $v(c_1) < 4$, or $v(c_2) < 2$.  So $v(c_1) = 3$
or $v(c_0) \in \{4, 5\}$.

After blowing up at $P$, the strict transform $D_1'$ of $D$ meets the
exceptional divisor $E_1$ completely in the chart $x = \pi s$, at the
point $Q$ given by $\pi = s = 0$.  On
this chart, we have $f = \pi^3(c_3s^3 + (c_2/\pi)s^2 + (c_1/\pi^2)s +
c_0/\pi^3)$.  So $D_1 = D_1' + E_1$ is cut out by $h := \pi(c_3s^3 + (c_2/\pi)s^2 + (c_1/\pi^2)s +
c_0/\pi^3)$.

Now, $$h \equiv \frac{c_1}{\pi}s + \frac{c_0}{\pi^2} \pmod{\mathfrak{m}_Q^4}.$$
From this we see that if $v(c_0) = 4$, then the multiplicity $m_2$ of $D_1$
at $x_2 := Q$ is $2$.  Otherwise, if $v(c_1) = 3$, then $m_2
= 3$, and the tangent cone is reducible.  In both of these cases, $D_1$ is robustly of
multiplicity $\leq 3$ at $Q$.  By
Lemma~\ref{Lminweierstrasspreserved}, the same is true for all further
$D_n$ at all points, and we are done.

Lastly, if $v(c_1) > 3$ but $v(c_0) = 5$, then $m_1 = 3$ and we take
one more blowup at the point $Q$.  After this blowup, the strict
transform $D_2'$ of $D_1$ meets the exceptional divisor $E_2$
completely in the chart $\pi = st$ at the point $R$ given by $s = t =
0$.  On this chart, we have that $D_2 = D_2' + E_2$ is cut out by
$j := st(c_3s^3 + (c_2/st)s^2 + (c_1/s^2t^2)s + c_0/s^3t^3)/s^2$.
Since $j \equiv c_3s^2t \pmod{\mathfrak{m}_R^4}$ (note that $\pi \in
\mathfrak{m}_R^2$), we see that, taking $x_3 = R$, we have $m_3 = 3$
  and the tangent cone to $D_2$ at $R$ is reducible.  So $D_2$ is
  robustly of multiplicity $\leq 3$ at $R$.  As in the previous cases,
  we are done.
  \end{proof}

\begin{proof}[Proof of Proposition~\ref{P:genus1}]
By Proposition~\ref{P:deg3}, it follows that there exists a minimal
Weierstrass equation $y^2=f(x)$ with $\deg(f)= 3$ such that $m_n \in
\{2,3\}$ for all $n$. Now change coordinates on $\mc{Y}_0 =
\mathbb{P}^1_{\mc{O}_K}$ by an element of $\GL_2(\mc{O}_K)$ (and also
all further $\mc{Y}_n$) to produce a new minimal Weierstrass equation
that satisfies Assumption~\ref{Afform} -- note that this invertible
change of variables does not affect any of the $m_n$. By Corollary~\ref{Cmain},
it then follows that if $\mc{X}_f$ is the normalization of the model $\mc{Y}_f$ in $K(X)$, then $-\Art(\mc{X}_f/\mc{O}_K) = \Delta_{X/K}$.  \qedhere
 \end{proof}
 
 \begin{remark}\label{R:multgoup}
The multiplicity of each singular point of $D_0$ being at most $3$ is
not sufficient to guarantee equality in every step of the induction. For instance, it is possible for $m_1=3$ and $m_2=4$ as the genus $2$ example $y^2= (x-1)(x-2)(x-3)(x-4t^2)(x-5t^2)(x-6t^2)$ over $K=\mathbb{C}((t))$ illustrates. In this case, $D_1 = D_0'+E$, and for calculating $m_2$ we also need to include the contribution coming from the multiplicity of the point on the exceptional divisor.
\end{remark}


\begin{bibdiv}
\begin{biblist}
 
\bib{BW_Glasgow}{article}{
  author={Bouw, Irene I.},
  author={Wewers, Stefan},
  title={Computing $L$-functions and semistable reduction of superelliptic
  curves},
  journal={Glasg. Math. J.},
  volume={59},
  date={2017},
  number={1},
  pages={77--108},
  issn={0017-0895},
}

\bib{BKSW}{article}{
   author={Bouw, Irene I.},
   author={Koutsianas, Angelos},
   author={Sijsling, Jeroen},
   author={Wewers, Stefan},
   title={Conductor and discriminant of Picard curves},
   journal={J. Lond. Math. Soc. (2)},
   volume={102},
   date={2020},
   number={1},
   pages={368--404},
}
		

\bib{DDMM}{article}{
   author={Dokchitser, Tim},
   author={Dokchitser, Vladimir},
   author={Maistret, C\'{e}line},
   author={Morgan, Adam},
   title={Arithmetic of hyperelliptic curves over local fields},
   journal={Math. Ann.},
   volume={385},
   date={2023},
   number={3-4},
   pages={1213--1322},
   issn={0025-5831},
   review={\MR{4566695}},
   doi={10.1007/s00208-021-02319-y},
}

\bib{Hartshorne}{book}{
   author={Hartshorne, Robin},
   title={Algebraic Geometry},
   series={Graduate Texts in Mathematics},
   volume={52},
  publisher={Springer-Verlag, New York-Berlin},
   date={1977},
   pages={xvi+496},
}


\bib{Kau}{article}{
   author={Kausz, Ivan},
   title={A discriminant and an upper bound for $\omega^2$ for hyperelliptic
   arithmetic surfaces},
   journal={Compositio Math.},
   volume={115},
   date={1999},
   number={1},
   pages={37--69},
   issn={0010-437X},
}

\bib{Kohls}{article}{
   author={Kohls, Roman},
   title={Conductors of superelliptic curves},
   date={2019},
   note={Ph.D. thesis, Universit\"{a}t Ulm},
}

\bib{Li:cd}{article}{
    AUTHOR = {Liu, Qing},
     TITLE = {Conducteur et discriminant minimal de courbes de genre {$2$}},
   JOURNAL = {Compositio Math.},
  FJOURNAL = {Compositio Mathematica},
    VOLUME = {94},
      YEAR = {1994},
    NUMBER = {1},
     PAGES = {51--79},
      ISSN = {0010-437X},
   MRCLASS = {14H45 (11G20 14H25)},
}

\bib{LiuBook}{book}{
    AUTHOR = {Liu, Qing},
     TITLE = {Algebraic geometry and arithmetic curves},
    SERIES = {Oxford Graduate Texts in Mathematics},
    VOLUME = {6},
      NOTE = {Translated from the French by Reinie Ern\'{e},
              Oxford Science Publications},
 PUBLISHER = {Oxford University Press, Oxford},
      YEAR = {2002},
     PAGES = {xvi+576},
      ISBN = {0-19-850284-2},
   MRCLASS = {14-01 (11G30 14A05 14A15 14Gxx 14Hxx)},
  MRNUMBER = {1917232},
MRREVIEWER = {C\'{\i}cero Carvalho},
}
\bib{Mau}{article}{
   author={Maugeais, Sylvain},
   title={Rel\`evement des rev\^{e}tements $p$-cycliques des courbes rationnelles
   semi-stables},
   language={French, with French summary},
   journal={Math. Ann.},
   volume={327},
   date={2003},
   number={2},
   pages={365--393},
   issn={0025-5831},
}


\bib{OS1}{article}{
   author={Obus, Andrew},
   author = {Srinivasan, Padmavathi},
   title={Conductor-discriminant inequality for hyperelliptic curves
     in odd residue characteristic},
   date={2021},
  eprint={arxiv:1910.02589v3},
}

 
\bib{ObusWewers}{article}{
   author={Obus, Andrew},
   author = {Wewers, Stefan},
   title={Explicit resolution of weak wild arithmetic surface
     singularities},
   journal={J.\ Algebraic Geom.}
   date={2020},
   volume={29},
   number={1},
   pages={691--728},
}

\bib{OS2}{article}{
   author={Obus, Andrew},
   author={Srinivasan, Padmavathi},
   title={Explicit minimal embedded resolutions of divisors on models of the
   projective line},
   journal={Res. Number Theory},
   volume={8},
   date={2022},
   number={2},
   pages={Paper No. 27, 27},
   issn={2522-0160},
   review={\MR{4409862}},
   doi={10.1007/s40993-022-00323-y},
}
\bib{saito2}{article}{
   author={Saito, Takeshi},
   title={Conductor, discriminant, and the Noether formula of arithmetic
   surfaces},
   journal={Duke Math. J.},
   volume={57},
   date={1988},
   number={1},
   pages={151--173},
   issn={0012-7094},
}

\bib{SerreLF}{book}{
   author={Serre, Jean-Pierre},
   title={Local fields},
   series={Graduate Texts in Mathematics},
   volume={67},
   note={Translated from the French by Marvin Jay Greenberg},
   publisher={Springer-Verlag, New York-Berlin},
   date={1979},
   pages={viii+241},
   isbn={0-387-90424-7},
}

\bib{PadmaRational}{article}{
  author = 	 {Srinivasan, Padmavathi},
  title = 	 {Conductors and minimal discriminants of
    hyperelliptic curves with rational Weierstrass points},
  eprint = {arxiv:1508.05172v1}, 
  year =         {2015}, 
}

\bib{PadmaTame}{article}{
  author = 	 {Srinivasan, Padmavathi},
  title = 	 {Conductors and minimal discriminants of
    hyperelliptic curves: a comparison in the tame case},
  eprint = {arxiv:1910.08228v1}, 
  year =         {2019}, 
}
	
\end{biblist}
\end{bibdiv}
\end{document}